\documentclass[12pt]{amsart}

\usepackage{fullpage}
\usepackage{verbatim}
\usepackage{url,graphicx, latexsym, url, color}
\usepackage{amssymb,amsmath,mathrsfs,amsthm}
\usepackage{ytableau}
\usepackage{easytable}

\newtheorem{theorem}{Theorem}
\newtheorem{lemma}[theorem]{Lemma}

\numberwithin{theorem}{section}

\newcommand{\CC}{\mathbb{C}}

\newcommand{\sgn}{\textup{sign}}

\newcommand{\Ind}{\textup{Ind}}
\newcommand{\Res}{\textup{Res}}
\newcommand{\wreath}{\textup{ wr }}
\newcommand{\mult}{\textup{mult}}

\title{The representation theory of generalized hyperoctahedral groups}
\author{William McGovern}
\address{William McGovern, Department of Mathematics,
University of Washington, Seattle, WA 98195}
\email{mcgovern@math.washington.edu} 
\author{James Pfeiffer}
\address{James Pfeiffer, Department of Mathematics,
University of Washington, Seattle, WA 98195}
\email{jamesrpfeiffer{@}gmail.com}

\begin{document}
\begin{abstract}
We give an explicit decomposition of $\Ind(1)_{B_n}^{S_{2n}}$, following Barbasch and Vogan \cite{bv}.
We define two natural generalizations of $B_n$, and extend the proof in \cite{bv} to recursively compute these decompositions.
Although the decompositions do not appear to follow a simple pattern, we prove enough of their structure to show that they are almost never multiplicity-free.
\end{abstract}
\maketitle

\section{Introduction} \label{BV:Introduction}
Let $B_n \subseteq S_{2n}$ be the {\em hyperoctahedral group}; that is, the stabilizer of $\sigma = (12)(34)\ldots (2n-1,2n)$ in $S_{2n}$. Barbasch and Vogan \cite{bv} showed that the induced representation $\Ind_{B_n}^{S_{2n}}(1)$ decomposes into a sum of Specht modules $\oplus_\lambda S^\lambda$, one for each $\lambda \vdash 2n$ such that each $\lambda_i$ is even.
We define two subgroups $C_{m,n}$ and $D_{m,n}$ of $S_{mn}$, each of which is a natural generalization of $B_n$. 
Let $C_{m,n} = S_n \wreath S_m$ and $D_{m,n} = S_n \wreath C_m$. Here $C_m \subseteq S_m$ is the cyclic group generated by an $m$-cycle. When $m=2$, $C_{m,n} = D_{m,n} = B_n$.

These generalizations arise naturally when using symmetry to reduce the dimension of semidefinite programs in combinatorial optimization. The $S_{2n}$-module $\Ind_{B_n}^{S_{2n}}(1)$ is naturally isomorphic to the vector space of perfect matchings on $K_{2n}$. Decomposing this vector space into irreducible representations corresponds to a block diagonalization of the semidefinite program underlying the {\em theta body} for these matchings, an approximation based on sums of squares.

Similarly, the $S_{mn}$-module $\Ind_{C_{m,n}}^{S_{mn}}(1)$ is naturally isomorphic to the vector space of perfect $m$-uniform hypermatchings on the $m$-uniform complete hypergraph $K^{(m)}_{mn}$. Likewise, $\Ind_{D_{m,n}}^{S_{mn}}(1)$ is naturally isomorphic to the vector space of decompositions of the vertex set $[mn]$ of the complete graph $K_{mn}$ into $n$ disjoint $m$-cycles. Decomposing these into irreducible representations would allow symmetry reduction of the corresponding combinatorial optimization problems.

We generalize Barbasch and Vogan's proof to recursively describe the decomposition of both $\Ind_{C_{m,n}}^{S_{mn}}(1)$ and $\Ind_{D_{m,n}}^{S_{mn}}(1)$. 
We do not believe that a simple pattern for the decomposition exists for $m > 2$ in either case.
However, we are able to establish enough of the structure of $\Ind_{C_{3,n}}^{S_{3n}}(1)$ to show that, unlike the case $m=2$,
the irreducible representations are not multiplicity-free for $n \ge 5$.

The structure of this paper is as follows. In Section 2, we give a method for determining $\Ind_{C_{m,n}}^{S_{mn}}(1)$ from $\Ind_{C_{m,n-1}}^{S_{m(n-1)}}(1)$. In Section 3, we produce an explicit linear isomorphism corresponding to the $m=2$ case.
In Section 4, we prove that $\Ind_{C_{m,n}}^{S_{mn}}(1)$ is not multiplicity-free for $n \ge 5$.

\section{Recursive construction}
We generalize the induction step in the proof of Barbasch and Vogan to the cases of $C_{m,n}$ and $D_{m,n}$.
First, we will recall the main ingredients in the case of $B_n$.

\begin{lemma}\label{BV:lemma1}
As homogeneous spaces, 
$S_{2n}/B_n \cong S_{2n-1}/(B_n \cap S_{2n-1}) = S_{2n-1}/B_{n-1}$.
\end{lemma}
\begin{proof}
The second equality follows from $B_n \cap S_{2n-1} = B_{n-1}$. For the first, define a map $\phi: S_{2n} / B_n \to S_{2n-1}/B_{n-1}$ by $\phi(gB_n) = (gB_n) \cap S_{2n-1}$. When defining $\phi$, choosing the coset representative $g \in S_{2n-1}$ shows that $\phi$ is well-defined. It's straightforward to check that the $S_{2n-1}$ action commutes with $\phi$.
\end{proof}

The next step lets us determine $\Ind_{B_n}^{S_{2n}}(1)$ by considering its restriction to $S_{2n-1}$. Although in general a representation is not uniquely determined by its restriction to a subgroup, we will see that in this case there is enough extra information to determine the decomposition.

\begin{lemma}\label{BV:lemma2}
The following recursive rule holds:
$$\Res_{S_{2n-1}}^{S_{2n}}\left(\Ind_{B_n}^{S_{2n}}(1)\right) = \Ind_{S_{2n-2}}^{S_{2n-1}}\left(\Ind_{B_{n-1}}^{S_{2n-2}}(1)\right).$$
\end{lemma}
\begin{proof}
By Lemma \ref{BV:lemma1}, $\Res_{S_{2n-1}}^{S_{2n}}(\Ind_{B_n}^{S_{2n}}(1)) = S_{2n-1}/B_{n-1}$. But this is just a restatement
of the definition of $ \Ind_{B_{n-1}}^{S_{2n-1}} (1) = \Ind_{S_{2n-2}}^{S_{2n-1}}(\Ind_{B_{n-1}}^{S_{2n-2}}(1))$.
\end{proof}

We will use the original result of Barbasch and Vogan in Section \ref{BV:section3}, so we prove it here for completeness. Here we say a partition $\lambda \vdash 2n$ is {\em even} if each of its parts $\lambda_i$ is even.

\begin{theorem}\label{BV:theorem1}
The decomposition of $\Ind_{B_n}^{S_{2n}}(1)$ into irreducibles is 
$$\Ind_{B_n}^{S_{2n}}(1) \cong \bigoplus_{\stackrel{\lambda \vdash 2n}{\lambda  \textup{ is even}}} S^\lambda.$$
\end{theorem}
\begin{proof}
This is true for $n=1$. We use induction. Assume $\Ind_{B_{n-1}}^{S_{2n-2}}(1)$ has the described decomposition. We use Lemma \ref{BV:lemma1}. By the branching rule, $\Ind^{S_{2n-1}}_{S_{2n-2}}(\Ind_{B_{n-1}}^{S_{2n-2}}(1))$ contains each $\mu \vdash 2n-1$ having exactly one odd part, and each such $\mu$ appears once. Suppose $\Ind_{B_n}^{S_{2n}}(1)$ contains $\lambda$ with at least three rows and at least two odd parts. Then the restriction of $\lambda$ contains a $\mu$ with at least two odd parts; thus these $\lambda$ do not occur. To rule out $\lambda=(\lambda_1,\lambda_2)$ with $\lambda_1$ and $\lambda_2$ odd, note that $(2n)$ occurs in $\Ind_{B_n}^{S_{2n}}(1)$ by Frobenius reciprocity. Therefore $(2n-1,1)$ can't occur in $\Ind_{B_n}^{S_{2n}}(1)$, as it would contribute a second copy of $(2n-1)$ to $\Ind^{S_{2n-1}}_{S_{2n-2}}(\Ind_{B_{n-1}}^{S_{2n-2}}(1))$. An induction on $i$ shows that $(2n-i,i)$ occurs in $\Ind_{B_n}^{S_{2n}}(1)$ if and only if $i$ is even.

Finally, consider a $\lambda$ with at least three even odd rows. Each $\mu$ obtained by deleting a box from $\lambda$ occurs in $\Ind^{S_{2n-1}}_{S_{2n-2}}(\Ind_{B_{n-1}}^{S_{2n-2}}(1))$ exactly once, and a single copy of $\lambda$ in $\Ind_{B_n}^{S_{2n}}(1)$ is the only way remaining to account for these $\mu$.
\end{proof}

We now generalize Lemmas \ref{BV:lemma1} and \ref{BV:lemma2} to the cases of $C_{m,n}$ and $D_{m,n}$.

\begin{lemma}\label{BV:lemma3}
The following two recursive rules hold:
\begin{align*}
\Res_{S_{nm-1}}^{S_{nm}}\left(\Ind_{C_{m,n}}^{S_{mn}}(1)\right) &= \Ind_{S_{(n-1)m} \times S_{m-1}}^{S_{nm-1}}\left(\Ind_{C_{m,n-1}}^{S_{m(n-1)}}(1) \otimes 1\right), \\
\Res_{S_{nm-1}}^{S_{nm}}\left(\Ind_{D_{m,n}}^{S_{mn}}(1)\right) &= \Ind_{S_{(n-1)m} }^{S_{nm-1}}\left(\Ind_{D_{m,n-1}}^{S_{m(n-1)}}(1)\right).
\end{align*}
\end{lemma}
\begin{proof}
The proof is a straightforward generalization of Lemmas \ref{BV:lemma1} and \ref{BV:lemma2}.
Observe that $C_{m,n} \cap S_{mn-1} = C_{m,n-1} \times S_{m-1}$ and that $D_{m,n} \cap S_{mn-1} = D_{m,n-1}$.
We then have that $S_{mn} / C_{m,n} \cong S_{mn-1} / (C_{m,n-1} \times S_{m-1})$ and 
$S_{mn} / D_{m,n} \cong S_{mn-1} / D_{m,n-1}$. The results follow.
\end{proof}
If we know the decomposition of $\Ind_{C_{m,n}}^{S_{mn}}(1)$ into irreducibles, we can use Lemma \ref{BV:lemma3} and Pieri's rule to decompose $\Ind_{C_{m,n+1}}^{S_{m(n+1)}}(1)$ into irreducibles. The same is true for $\Ind_{D_{m,n}}^{S_{mn}}(1)$ and $\Ind_{D_{m,n+1}}^{S_{m(n+1)}}(1)$, except that we use the branching rule. See Table \ref{BV:table1} for some results for $m=3$ and small $n$.

\section{Explicit isomorphism for $m=2$}\label{BV:section3}
Recall that a {\em matching} in a graph is a set of disjoint edges; we say a matching is a $k$-matching if it consists of $k$ edges. Take $S$ to be the set of $n$-matchings in $K_{2n}$; these are also known as perfect matchings. If we let $S_{2n}$ permute the vertices of $K_{2n}$, then $S$ is an $S_{2n}$-set and $\CC[S]$ an $S_{2n}$-module. Note that $S_{2n}$ acts transitively on $S$. 

Fix the matching $s = 12|34|\cdots|2n-1,2n$. Then the stabilizer of $s$ in $S_{2n}$ is exactly $B_n$ as defined in Section 
\ref{BV:Introduction}. Then $\Ind_{B_n}^{S_{2n}}(1) \cong \CC[S]$ as $S_{2n}$-modules. We give an explicit decomposition of $\CC[S]$ into irreducibles; i.e., we provide a concrete linear map from each summand $S^\lambda \to \CC[S]$. Note that the decomposition is determined up to isomorphism by Theorem \ref{BV:theorem1}. Our contribution here is to give an effectively computable isomorphism.

\begin{lemma}\label{hyperoctahedral}
Let $S$ be the set of $k$-matchings in $K_{2k}$. Then $\CC[S] \cong \bigoplus_{\lambda} S^\lambda$, where the direct sum is over all partitions $\lambda$ of $2k$ consisting of even parts. The multiplicity of each $S^\lambda$ is 1.
\end{lemma}
\begin{proof}
Fix an even $\lambda$. We will define a map $f: M^\lambda \to \CC[S]$. For a single-row tabloid $R$, let $f(R)$ be the sum of all matchings in $R$. For a tabloid $T$ with rows $R_i$, let $f(T) = \prod_i f(R_i)$; we interpret the product of disjoint matchings as their union. For example:
\begin{align*}
 \ytableausetup{tabloids} 
f\left(\ytableaushort{1234}\right) &= 12|34 + 13|24 + 14|23 \\
f\left(\ytableaushort{1234,56}\right) &= f([1234])f([34]) \\ &= (12|34 + 13|24 + 14|23) \cdot (56) \\ &= 12|34|56 + 13|24|56 + 14|23|56
\end{align*}
Extend by linearity to $M^\lambda$. This is a map of $S_n$-modules, so its restriction to $S^\lambda$ is either 0 or an isomorphism.

Let $t$ be the standard tableau with entries in increasing order. We will show $f(e_t) \ne 0$. $f(\{t\})$ contains the term $m=12|34|\cdots|2k-1,2k$. If $\pm\pi\{t\}$ is another term in $e_t$ such that $f(\pi\{t\})$ also contains $m$, then $2i$ and $2i-1$ must be in the same row of $\pi\{t\}$ for all $i$. But using column group operations, this is only possible if we switch $2i$ with $2j$ and $2i-1$ with $2j-1$. Therefore $\pi$ is a product of an even number of disjoint transpositions, and in particular, $\sgn (\pi) = 1$. So $f(\{t\})$ appears with positive sign in $f(e_t)$, and therefore $f(e_t) \ne 0$.

The proof is completed by noting that, per Theorem \ref{BV:theorem1}, we have accounted for each irreducible representation that appears.
\end{proof}

\section{Multiplicities occur for $m=3, n \ge 5$}
The recursion rules established in Lemma \ref{BV:lemma3} can be used to compute the decompositions of 
$\Ind_{C_{m,n}}^{S_{mn}}(1)$ and 
$\Ind_{D_{m,n}}^{S_{mn}}(1)$ for small values of $m$ and $n$; see Table \ref{BV:table1} at the end of this section. 

As discussed above, in all cases we computed, there was a unique
solution to the recursion containing a copy of the trivial representation.
However, unlike the case $m=2$, there does not seem to be any simple pattern to the decomposition.
In particular, the decompositions are not multiplicity-free after the first few values of $n$.

In this section, we consider $V_n := \Ind_{C_{3,n}}^{S_{3n}}(1)$, and determine enough of the structure of $V_n$ to show that
for $n \ge 5$, $V_n$ is not multiplicity-free. We accomplish this by considering {\em partition patterns}. A partition pattern $\lambda = (*,\lambda_1,\ldots,\lambda_k)$ represents any partition of length $k+1$ whose second through last parts equal $\lambda$. We also abandon tuple notation and simply concatenate digits, as all our entries are at most 9. For instance, the partition pattern $42$ represents the partition $(n-6,4,2)$ for any $n$. As a special case, we let $0$ denote the pattern $\emptyset$, representing the partition $(n)$ for any $n$.

For any partition pattern $\lambda$, let $\mult(\lambda,n)$ be the multiplicity of $S^\lambda$ in $V_n$. Also let $\mult(\lambda,n^-)$ be the multiplicity of $S^\lambda$ in $\Res^{S_{3n}}_{S_{3n-1}}(V_n) = \Ind_{S_{3n-3}}^{S_{3n-1}}(V_{n-1})$. It is also convenient to refer to $V_n$ and $\Res^{S_{3n}}_{S_{3n-1}}(V_n)$ as {\em level $n$} and {\em level $n^-$}, respectively. 

We will first determine the multiplicities of certain $S^\lambda$ in $V_n$. Then, we will use this structure to show that $V_n$ is not multiplicity-free for
$n \ge 5$.

\begin{lemma}\label{BV:inductivelemma}
The following $\lambda$ have multiplicity $1$ in all levels $n \ge 5$: 
$0,2,3,4,22,5,41,32$. The following $\lambda$ do not appear in any level: $1,21,31,221,311,411$.
\end{lemma}
\begin{proof}
It is easy to check that this holds for $n=5$; see Table \ref{BV:table1}. Assuming by induction that the
given decomposition holds for $n-1$, we get a partial list of multiplicities
at level $n^-$:

$$
\begin{array}{c||c|c|c|c|c|c|c|c|c|c|c|c|c|c|c|c|c}
\lambda &0&1&11&2&21&3&111&4&31&22&211&1111&5&41&32&311&221 \\ \hline
\mult(\lambda,n^-)&1&1&0&2&1&2&0&3&2&2&0&0&3&3&3&0&1
\end{array}
$$
It is then straightforward to check that
the given decomposition for level $n$ is the only way to recover these multiplicities at level $n^-$.
\end{proof}

\begin{theorem}\label{BV:multiplicities}
All levels $V_n$ for $n \ge 5$ have multiplicities.
\end{theorem}
\begin{proof}
By Lemma \ref{BV:inductivelemma}, it follows that $\mult(51,n) + \mult(42,n) = 2$ for all $n$.
By considering the relevant children at level $n-1$, we can see that $\mult(51,n^-) + \mult(42,n^-) = 9$.
 Therefore, one of
$\mult(51,n^-)$, $\mult(42,n^-) \ge 5$. But since each of $51$ and $42$ has four parents at
level $n$, we must have multiplicities at level $n$.
\end{proof}

\begin{table}[h]
\caption{The decomposition of $\Ind_{C_{m,n}}^{S_{mn}}(1)$ into irreducible representations for $m=3$ and small $n$. Note that $n=5$ is the first to contain multiplicities.}
\begin{center} 
    \begin{tabular}{|c|c|p{12cm}|}
    \hline
    $n$ & $m$ & $\Ind_{C_{m,n}}^{S_{mn}}(1)$ \\ \hline
    2 & 3 & $[4, 2]$, $[6]$ \\ \hline
    3 & 3 & $[4, 4, 1]$, $[5, 2, 2]$, $[6, 3]$, $[7, 2]$, $[9]$  \\ \hline
    4 & 3 & $[4, 4, 4]$, $[5, 4, 2, 1]$, $[6, 2, 2, 2]$, $[6, 4, 2]$, $[6, 6]$, $[7, 3, 2]$, $[7, 4, 1]$, $[8, 2, 2]$, $[8, 4]$, $[9, 3]$, $[10, 2]$, $[12]$ \\ \hline
    5 & 3 & $[5, 4, 4, 2]$, $[5, 5, 3, 1, 1]$, $[6, 4, 2, 2, 1]$, $[6, 4, 4, 1]$, $[6, 5, 2, 2]$, $[6, 6, 3]$, $[7, 2, 2, 2, 2]$, $[7, 4, 2, 2]$, $[7, 4, 3, 1]$, $[7, 4, 4]$, $[7, 5, 2, 1]$, $[7, 6, 2]$, $[8, 3, 2, 2]$, $[8, 4, 2, 1]$, $[8, 4, 3]$, $[8, 5, 2]$, $[8, 6, 1]$, $[9, 2, 2, 2]$, $[9, 4, 2]$, $[9, 4, 2]$, $[9, 6]$, $[10, 3, 2]$, $[10, 4, 1]$, $[10, 5]$, $[11, 2, 2]$, $[11, 4]$, $[12, 3]$, $[13, 2]$, $[15]$ \\ \hline
    \end{tabular}
\label{BV:table1}
\end{center}
\end{table}

\bibliographystyle{plain}
\bibliography{young}

\end{document}